\documentclass[11pt,a4paper,reqno]{amsart}
\usepackage{amsmath}
\usepackage{mathrsfs}
\usepackage{amsfonts}
\usepackage{amssymb}
\usepackage{graphicx}
\usepackage{amsthm}
\usepackage{tikz}
\usepackage{pgfplots}
\usepackage[latin2]{inputenc}
\usepackage{t1enc}
\usepackage{multicol, verbatim}
\newcommand\C{\mathbb{C}}

\newcommand\R{\mathbb{R}}

\newcommand\D{\mathscr{D}}
\newcommand\ST{\mathscr{S}}

\numberwithin{equation}{section}

\usepackage{hyperref}
\usepackage{xcolor}
\hypersetup{
    colorlinks,
    linkcolor={red!50!black},
    citecolor={blue!50!black},
    urlcolor={blue!80!black}
}
\definecolor{ao(english)}{rgb}{0.0, 0.5, 0.0}

\theoremstyle{plain}
\theoremstyle{remark}
\newtheorem{remark}{Remark}[section]
\theoremstyle{lemma}
\newtheorem{lemma}{Lemma}[section]
\newtheorem{theorem}{Theorem}[section]
\theoremstyle{corollary}
\newtheorem{corollary}[theorem]{Corollary}

\usepackage{mathtools}

\DeclarePairedDelimiter\floor{\lfloor}{\rfloor}

\title{Limit Shapes for Unimodal Sequences}
\date{}
\author{Walter Bridges}

\address{Louisiana State University \\ Department of Mathematics }

\email{wbridg6@lsu.edu}



\begin{document}

\begin{abstract}
We prove asymptotic 0-1 Laws satisfied by diagrams of unimodal sequences of positive integers.  These diagrams consist of columns of squares in the plane, and the upper boundary is called the shape.  For various types, we show that, as the number of squares tends to infinity, $100\%$ of shapes are near a certain curve---that is, there is a single {\it limit shape}.  Similar phenomena have been well-studied for integer partitions, so the present work is a natural extension.  One notable corollary is a transferred limit shape for overpartitions.
\end{abstract}

\maketitle

\section{Introduction}
\subsection{Statement of Main Results}
A {\it unimodal sequence} $\lambda=\{\lambda_j\}_{j=1}^s$ of {\it size} $n$ is a sequence of positive integers that sum to $n$ and that increase weakly and then decrease weakly:
\begin{equation}\label{unimodaldef}
\lambda: \qquad 0<\lambda_1 \leq \dots \leq \lambda_{k-1} \leq \lambda_k \geq \lambda_{k+1} \geq \dots \geq \lambda_s > 0 \qquad \text{and} \qquad \sum_{j=1}^s \lambda_j=n.
\end{equation}
For example, the unimodal sequences of size $4$ are $(1,1,1,1)$, $(1,1,2)$, $(1,2,1)$, $(2,1,1)$, $(2,2)$, $(1,3)$, $(3,1)$, $(4).$  The first systematic study of unimodal sequences and their asymptotic behavior is usually credited to Wright, who called them ``stacks'' in a series of papers \cite{W1}-\cite{W3}.  

We refer to the $\lambda_j$ as the {\it parts} of a sequence.  Its {\it peaks} are $\lambda_k$ and any other parts equal to $\lambda_k$.  The {\it diagram} of $\lambda$ is the set of adjacent columns of unit squares in the plane, where the $j$-th column has $\lambda_j$ squares.  To fix a centering of a diagram, we will always choose to place the left-most peak vertex on the $y$-axis (although our results hold regardless of which peak vertex we fix as the center).  In this paper, we study the {\it shape}, $\varphi(\lambda)$, which is the top border of the diagram of $\lambda$.

\begin{figure}
\begin{tikzpicture}[scale=.7, line width=1pt]
  \draw (0,0) grid (7,1);
  \draw (1,2) grid (6,1);
  \draw (2,3) grid (5,1);
   \draw (2,4) grid (5,3);
    \draw (3,5) grid (5,3);
    \draw[->] (3,0) -- (3,7);
    \draw[->] (0,0) -- (8, 0);
    \draw[<-] (-2,0) -- (0, 0);
    \draw[color=red, line width= 2pt] (0,0) -- (0,1) -- (1,1) -- (1,2) -- (2,2) -- (2,4)-- (3,4) -- (3,5)-- (5,5) -- (5,2) -- (6,2) -- (6,1) -- (7,1) -- (7,0);
    \draw  (1,4) -- (1,4) node[anchor=east] {\Large {\color{red} $\widetilde{\varphi}(\lambda)$}};
    \draw (0,0) -- (0,0) node[anchor=north] {$\frac{-3}{\sqrt{20}}$};
   \draw (3 ,5) -- (3 ,5) node[anchor=east] {$\frac{5}{\sqrt{20}}$};
   \draw (7 ,0) -- (7 ,0) node[anchor=north] {$\frac{4}{\sqrt{20}}$};
 \end{tikzpicture}
\caption{Diagram and renormalized shape for $\lambda=(1,2,4,5,5,2,1)$ of size $20$.}
\label{varphilambda}
\end{figure}
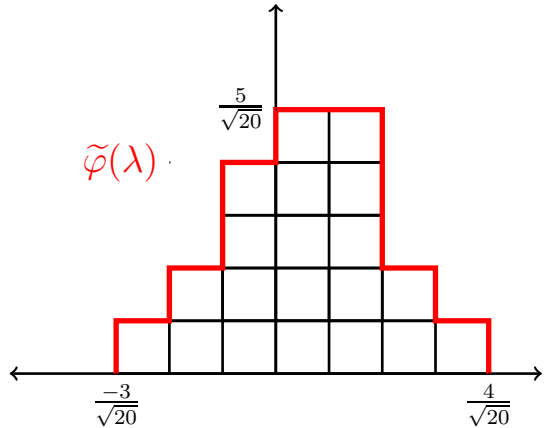

To compare diagrams of large size to a fixed curve, it is convenient to rescale them to have area 1, so let us define the {\it renormalized shape} $\widetilde{\varphi}(\lambda)$ to be the shape obtained from $\varphi(\lambda)$ by rescaling both the $x$- and $y$-axes by $\frac{1}{\sqrt{n}}$ when the size of $\lambda$ is $n$ (see Figure \ref{varphilambda}).

Roughly, the question we want to answer is the following: What are the typical shapes of diagrams of size $n$, as $n \to \infty$?  Here, ``typical'' will mean ``under the uniform probability measure on diagrams of size $n$''.  With these notions of ``typical shape'', it will turn out that, for the types of unimodal sequences we consider, there is a single {\it limit shape}.

This type of striking 0-1 law has been well-studied for integer partitions.  In \cite{F}, Fristedt introduced probabilistic machinery to make a deep study of the limiting behavior of partitions.  This machinery was subsequently used by Vershik in \cite{V} to state many types of limit shapes.  Proofs of limit shapes for unrestricted and distinct parts partitions, by means of a stronger large deviation principle, were finally proved by Dembo, Vershik and Zeitouni in \cite{DVZ}.  For more information on the history of limit shapes for partitions, we direct the reader to Sections 1 and 12 of \cite{DP}.  

In particular, a limit shape for unrestricted partitions of size $n$ under the uniform probability measure is
\begin{equation}\label{Vershikscurve}
y=f_p(x)=-\frac{\sqrt{6}}{\pi}\log\left(1-e^{-\frac{\pi}{\sqrt{6}}x}\right).
\end{equation}
Note that this can be symmetrized as $e^{\frac{\pi}{\sqrt{6}}x}+e^{\frac{\pi}{\sqrt{6}}y}=1$, which respects the involution on partitions given by conjugation.  (See \cite{A98} \S 1.3.)  An ``elementary'' proof (one that does not require measure theory) that \eqref{Vershikscurve} is the limit shape for partitions was given by Petrov in \cite{P}, and we will take such an approach here.  A different type of problem was recently solved by DeSalvo and Pak, who found conditions under which partition bijections allow for the transfer of limit shapes \cite{DP}.  We will see one result of this type in the present paper.

We now state our main results.  Following the notation of Bringmann-Mahlburg in \cite{BM}, let $\ST(n)$ denote the set of (unrestricted) unimodal sequences of size $n$, and denote its cardinality by $s(n)$.  Let $\D(n)$ denote the set of {\it strongly unimodal sequences} of size $n$, and denote its cardinality by $d(n)$; these have the added requirement that all of the inequalities in \eqref{unimodaldef} are strict.  Finally, let $\D_m(n)$ denote the set of {\it semi-strict unimodal sequences} of size $n$, and denote its cardinality by $dm(n)$; here, we require that there be a single peak and that the inequalities to the left of it in \eqref{unimodaldef} are strict.

For a function $f(x)$, let $N_{\epsilon}(f)$ denote the set of points in the plane whose horizontal distance from $y=f(x)$ is at most $\epsilon$, together with $\epsilon$ neighborhoods of the $x$- and $y$-axes.  (The latter components of $N_{\epsilon}$ are necessary to account for vertical and horizontal asymptotes of functions.)

\begin{theorem}[Strongly Unimodal Sequences]\label{T:D}
Let $\epsilon > 0$ be arbitrary and let
$$
f_d(x):= \begin{cases} - \frac{\sqrt{6}}{\pi} \log \left(e^{-\frac{\pi}{\sqrt{6}}x}-1 \right) & \text{if $x \in \left[-\frac{\sqrt{6}}{\pi}\log(2),0 \right)$,} \\ - \frac{\sqrt{6}}{\pi} \log \left(e^{\frac{\pi}{\sqrt{6}}x}-1 \right) & \text{if $x \in \left(0,\frac{\sqrt{6}}{\pi}\log(2) \right]$.} \end{cases}
$$
Then
\begin{equation}\label{shapeasym} \lim_{n \to \infty} \frac{1}{d(n)} \cdot \# \left\{ \lambda \in \D(n): \widetilde{\varphi}(\lambda) \subset N_{\epsilon}(f_d) \right\} = 1. \end{equation}
\end{theorem}

Note that $$
\int_{\R} f_{d}(x)dx= 2\int_0^{\frac{\sqrt{6}}{2} \log 2} -\frac{\sqrt{6}}{\pi} \log\left(e^{\frac{\pi}{\sqrt{6}}x}-1\right)dx= \frac{12}{\pi^2}\int_1^2 \frac{\log(t-1)}{t}dt
$$
$$
=\frac{12}{\pi^2}\left(-\text{Li}_2(1-t)-\log(t-1)\log t\right) \Big\rvert_{t=1}^2= -\frac{12}{\pi^2}\text{Li}_2(-1)=1,
$$
where, as usual, the dilogarithm is defined for $z \in \C \setminus \R_{>1}$ by  $\frac{d}{dz} \text{Li}_2(z)= -\frac{1}{z} \log(1-z)$ and $\text{Li}_2(0)=0$.  Here, $\log(z)$ is the principal branch of the complex logarithm.  For $|z|\leq 1$, one also has the series representation, $\text{Li}_2(z)=\sum_{n \geq 1} \frac{z^n}{n^2}$.  Integrals of the other functions below are similarly evaluated in terms of $\text{Li}_2(z)$.

\begin{theorem}[Unrestricted Unimodal Sequences]\label{T:S}  Let $\epsilon > 0$ be arbitrary and let
$$
f_s(x):= \begin{cases} -\frac{\sqrt{3}}{\pi} \log \left(1-e^{\frac{\pi}{\sqrt{3}}x} \right) & \text{if $x < 0$,} \\  -\frac{\sqrt{3}}{\pi} \log \left(1-e^{-\frac{\pi}{\sqrt{3}}x} \right) & \text{if $x > 0$.} \end{cases}
$$
Then
$$
\lim_{n \to \infty} \frac{1}{s(n)} \cdot \# \left\{ \lambda \in \ST(n) :\widetilde{\varphi}(\lambda) \subset N_{\epsilon}(f_s) \right\}=1.
$$
\end{theorem}

\begin{remark}
The limit shape of Theorem \ref{T:S} also holds for ``unimodal sequences with summits'', which are distinguished from unrestricted unimodal sequences by designating one peak as the ``summit''.  These were called ``stacks with summits'' in \cite{BM}, and the number of unimodal sequences with summits was denoted by $ss(n)$.  In particular, we have $ss(n) \sim s(n)$ (see \cite{BM}).  It is straightforward to repeat our proof of Theorem \ref{T:S} for unimodal sequences with summits with very little change.
\end{remark}

\begin{remark}
It is not surprising that the limit shapes for unrestricted and strongly unimodal sequences are made of two halves of the limit shapes for unrestricted and distinct parts partitions (see \cite{V}, Th. 4.4 and Th. 4.5).  In particular, each half of the limit shape for unrestricted unimodal sequences is the curve \eqref{Vershikscurve} scaled down so that the area beneath it is $\frac{1}{2}$.  \end{remark}

\begin{theorem}[Semi-strict Unimodal Sequences]\label{T:DM}
Let $\epsilon > 0$ be arbitrary and let
$$
f_{dm}(x):= \begin{cases} -\frac{2}{\pi}\log\left(e^{-\frac{\pi}{2}x}-1\right) & \text{if $x \in \left[-\frac{2}{\pi}\log 2, 0\right)$,} \\ -\frac{2}{\pi}\log\left( 1-e^{-\frac{\pi}{2}x}\right) & \text{if $x >0$.} \end{cases}
$$
Then
$$
\lim_{n \to \infty} \frac{1}{dm(n)} \cdot \#\left\{\lambda \in \D_m(n) : \widetilde{\varphi}(\lambda) \subset N_{\epsilon}(f_{dm})\right\} = 1.
$$
\end{theorem}
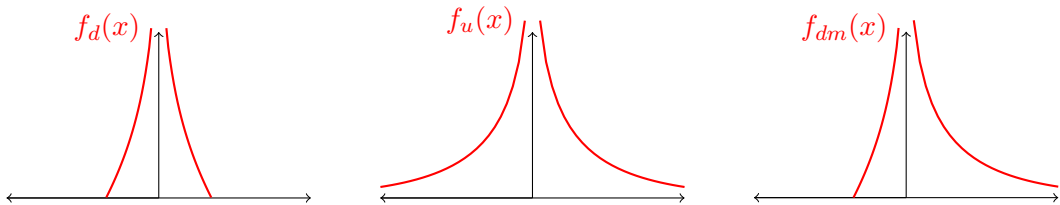
\begin{figure}
\begin{tikzpicture}[scale=1,domain=-2:2]

    \draw[->] (-2,0) -- (2,0);
    \draw[->] (0,0) -- (0,2.2);
    \draw[<-] (-2,0) -- (0,0);
    \draw[thick, color=red]   plot[domain=.1:.69] (\x,{  -ln(exp(\x )-1)})   ;
    \draw[thick, color=red]   plot[domain=-.69:-.1] (\x,{  -ln(exp(-\x )-1)})  node[left] {$f_d(x)$} ;
\end{tikzpicture} \qquad
\begin{tikzpicture}[scale=1,domain=-2:2]

    \draw[->] (-2,0) -- (2,0);
    \draw[->] (0,0) -- (0,2.2);
    \draw[<-] (-2,0) -- (0,0);
    \draw[thick, color=red]   plot[domain=.1:2] (\x,{  -ln(1-exp(-\x ))})   ;
    \draw[thick, color=red]   plot[domain=-2:-.1] (\x,{  -ln(1-exp(\x ))})  node[left] {$f_u(x)$} ;
\end{tikzpicture}  \qquad
\begin{tikzpicture}[scale=1,domain=-2:2]

    \draw[->] (-2,0) -- (2,0);
    \draw[->] (0,0) -- (0,2.2);
    \draw[<-] (-2,0) -- (0,0);
    \draw[thick, color=red]   plot[domain=.1:2] (\x,{  -ln(1-exp(-\x))})    ;
    \draw[thick, color=red]   plot[domain=-ln(2):-.1] (\x,{  -ln(exp(-\x)-1)}) node[left]  {$f_{dm}(x)$} ;
\end{tikzpicture}
\caption{Respective limit shapes for strongly, unrestricted and semi-strict unimodal sequences}
\label{limitshapes}
\end{figure}

\begin{remark} 
We observe that the left-half of $f_{dm}(x)$ is the limit shape for distinct parts partitions scaled so that the area beneath it is $\frac{1}{3}$ (\cite{V}, Th. 4.5), while the right half is the limit shape for unrestricted partitions scaled so that the area beneath it is $\frac{2}{3}$.  We discuss the appearance of these constants further in Section \ref{S:C}.
\end{remark}

Our proofs of the main results are structured as follows.  Following a method of Petrov (\cite{P}, \S 6), we will obtain limit shapes for the left and right halves ``in isolation'', showing that as $n \to \infty$, $0\%$ of left (resp. right) halves of shapes are {\it not} in an $\epsilon$ neighborhood of some left (resp. right) limit shape.  For Theorem \ref{T:D}, this is enough to complete the proof.  However, for Theorems \ref{T:S} and \ref{T:DM}, we will need to analyze peaks more closely; we will show that, on average, peaks are $\omega(\sqrt{n})$, so that a degenerate limit shape without a vertical asymptote does not occur.

\subsection{Some Consequences of Theorems \ref{T:D}-\ref{T:DM}}\label{S:C}

 The most natural consequences of Theorems \ref{T:D}-\ref{T:DM} concern the structure of unimodal sequences of size $n$ at the scale of $\sqrt{n}$.  For example, Theorem \ref{T:D} implies the following corollary concerning the number of parts in strongly unimodal sequences.
 
\begin{corollary}
Let $\epsilon > 0$ be arbitrary.  The number of parts of $100\%$ of strongly unimodal sequences of size $n$ as $n \to \infty$ lies in the interval $\sqrt{n}\left(2\frac{\sqrt{6}}{\pi}\log 2 - \epsilon, 2\frac{\sqrt{6}}{\pi}\log 2 + \epsilon \right)$.
\end{corollary}

We leave the statement of similar corollaries to the reader.

Recall that the {\it rank} of a semi-strict unimodal sequence is the number of parts to the right of the peak minus the number of parts to the left of the peak.  Bringmann--Jennings-Shaffer--Mahlburg proved that the limiting distribution of this statistic is a point mass with mean $\frac{\sqrt{n}\log n}{\pi}$ (\cite{BJSM}, Prop. 1.2 part (3)).  Theorem \ref{T:DM} anticipates this result. 

After Theorem \ref{T:DM}, we see that a typical semi-strict unimodal sequence of size $n$ is made up of a distinct parts partition of size roughly $\frac{n}{3}$ and an unrestricted partition of size roughly $\frac{2n}{3}$.  It follows from Theorem \ref{T:DM} that this distinct parts partition has roughly $\frac{\sqrt{6}}{\pi}\log (2) \sqrt{n}$ parts.  Erd\H{o}s-Lehner proved that a typical partition of size $m$ has roughly $\frac{\sqrt{3}}{\pi\sqrt{2}}\sqrt{m}\log m$ parts (\cite{EL}, Th. 1.1).  Hence, we should expect the limiting rank of $100\%$ semi-strict unimodal sequences to be 
$$
\frac{\sqrt{3}}{\pi\sqrt{2}}\sqrt{\frac{2n}{3}}\log \left(\frac{2n}{3}\right)-\frac{\sqrt{6}}{\pi}\log (2) \sqrt{n} \sim \frac{\sqrt{n}\log n}{\pi},
$$
as proved in \cite{BJSM} using the Method of Moments.

Theorem \ref{T:DM} also leads to a limit shape for overpartitions, combinatorial objects having many similarities to classical partitions.  As defined in \cite{CL}, an overpartition is a partition in which the last occurrence of a part may (or may not) be marked.  For example, the overpartitions of size $3$ are $(\overline{3})$, $(3)$, $(\overline{2}, \overline{1})$, $(\overline{2}, 1)$, $(2, \overline{1})$, $(2,1)$, $(1,1,\overline{1})$, $(1,1,1)$.  We denote the set of overpartitions of size $n$ by $\overline{\mathscr{P}}(n)$ with cardinality $\overline{p}(n)$.

From \cite{A13} equation 1.7, we have the generating function identity
$$
\frac{1+q}{q}\sum_{n \geq 1} dm(n)q^n=\prod_{j \geq 1} \frac{1+q^j}{1-q^j}= \sum_{n \geq 0} \overline{p}(n)q^n,
$$
thus $dm(n+1)+dm(n)=\overline{p}(n)$.  We now give a short bijective proof of this equality and use it to derive a limit shape for overpartitions.  If $\lambda \in \D_m(n)$, let the peak and parts to its left be marked.  If $\lambda \in \D_m(n+1)$, let only the parts to its left be marked and subtract 1 from the peak.

If we plot diagrams for overpartitions as Vershik does for partitions in \cite{V}---in the first quadrant as weakly decreasing columns of squares and without distinguishing marked parts---then our bijection leads to a map between diagrams of semi-strict unimodal sequences and a transfer of limit shapes.  Since $dm(n+1) \sim dm(n)$, it is easy to see that a limit shape for overpartitions is obtained immediately by adding horizontal components of the limit shape for semi-strict unimodal sequences.

\begin{corollary}[Overpartitions]\label{C:Pbar}
Let $\epsilon > 0$ be arbitrary and let $(g)^{-1}$ denote the inverse function of $g$, so
$$
f_{\overline{p}}(x):= \left(\left(-\frac{2}{\pi}\log\left( 1-e^{-\frac{\pi}{2}x}\right)\right)^{-1}-\left(-\frac{2}{\pi}\log\left(e^{-\frac{\pi}{2}x}-1\right)\right)^{-1}\right)^{-1}=  \frac{2}{\pi} \log\left(\frac{1+e^{-\frac{\pi}{2}x}}{1-e^{-\frac{\pi}{2}x}}\right).
$$
Then
$$
\lim_{n \to \infty} \frac{1}{\overline{p}(n)} \left\{\lambda \in \overline{\mathscr{P}}(n) : \widetilde{\varphi}(\lambda) \subset N_{\epsilon}(f_{\overline{p}}) \right\} = 1.
$$
\end{corollary}
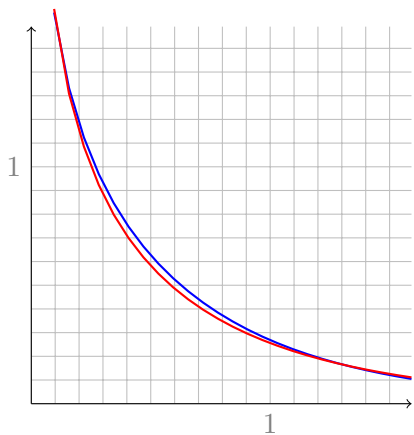
\begin{figure}
 \begin{tikzpicture}[scale=2,domain=0:2]

    \draw[thin, color=lightgray, step=.1570796] (0,0) grid (2.5,2.5);
    \draw[thin, color=gray] (1.570796,0) node[below] {$1$};
    \draw[thin, color=gray] (0,1.570796) node[left] {$1$};
    \draw[->] (0,0) -- (2.5,0);
    \draw[->] (0,0) -- (0,2.5);
   \draw[thick, color=blue]   plot[domain=.15:2.5] (\x,{  ln((1+exp(-\x))/(1-exp(-\x)))}) ;
   \draw[thick, color=red]   plot[domain=.15:2.5] (\x,{  1.2*-ln(1-exp(-.8*\x))});
    \end{tikzpicture}
    \caption{Limit shapes for overpartitions, {\color{blue} $f_{\overline{p}}(x)$}, and unrestricted partitions, {\color{red} $f_p(x)$}}
    \label{overpartitionshape}
    \end{figure}

If we represent marked parts in a diagram by shading the top square, we see that conjugation is also an involution on overpartitions.  And indeed, one checks that $y=f_{\overline{p}}(x)$ is symmetric in $x$ and $y$.

\begin{remark}\label{R:overpartitionsweight}
In \cite{CH}, Corteel-Hitczenko proved that the expected weight of overlined parts is asymptotic to $\frac{n}{3}$.  In view of Theorem \ref{T:DM} and the map between semi-strict unimodal sequences and overpartitions, we obtain the following refinement: For any $\epsilon > 0$ and ``for $100\%$ of overpartitions as $n \to \infty$'', the total weight of marked parts lies between $\frac{n}{3}-\epsilon\sqrt{n}$ and $\frac{n}{3}+\epsilon\sqrt{n}$.
\end{remark}

\begin{remark}
In Figure \ref{overpartitionshape}, we see that the limit shapes $f_p(x)$ and $f_{\overline{p}}(x)$ intersect at $x=.1398 \dots$ and $x=1.4088\dots$.   It would be interesting to give a more direct combinatorial explanation that overpartitions tend to ``bulge out'' more than partitions near the middle of their diagrams.  And more generally, it would be interesting to see a systematic study of how to alter/restrict partitions to achieve a desired geometric effect on the limit shape.
\end{remark}

The rest of the paper is structured as follows.  In Section \ref{S:notation}, we briefly introduce some notation.  In Section \ref{S:D}, we give a full proof of Theorem \ref{T:D}.  In Sections \ref{S:S} and \ref{S:DM}, respectively, we outline the proofs of Theorems \ref{T:S} and \ref{T:DM}, highlighting an additional technical difficulty that does not appear in Section \ref{S:D}.

\section*{Acknowledgements}

The author thanks Igor Pak for clarifying some of the history of limit shapes for partitions.

\section{Notation}\label{S:notation}

A renormalized shape $\widetilde{\varphi}(\lambda)$ consists of line segments that meet at $90^{\circ}$ corners, which we will call {\it vertices} (see Figure \ref{vertices}).  Let $V_{p}(\lambda)$ be the set of {\it peak vertices}, namely those bounding the top edge of $\widetilde{\varphi}(\lambda)$.  Let $V_{\ell}(\lambda)$ be the set of left vertices to the left of peak vertices, but not including the vertex on the $x$-axis.  Similarly, let $V_r(\lambda)$ be the set of right vertices.  Recall that $N_{\epsilon}$ contains an $\epsilon$-neighborhood of the $x$- and $y$-axes.  Thus, for $f \in \{f_d, f_{s}, f_{dm}\}$, we have $\widetilde{\varphi}(\lambda) \subset N_{\epsilon}(f)$ if and only if $V_{\ell}(\lambda) \cup V_r(\lambda) \cup V_p(\lambda) \subset N_{\epsilon}(f)$, and to prove Theorem \ref{T:D} it suffices to show
\begin{equation}\label{E:Vjlim}
\lim_{n \to \infty} \frac{1}{d(n)}\cdot \#\left\{ \lambda \in \D(n) : V_{j}(\lambda) \not\subset N_{\epsilon}(f_d)\right\} = 0,
\end{equation}
where $j \in \{\ell, r, p\}$, and we may prove Theorems \ref{T:S} and \ref{T:DM} similarly. \\ \

\begin{figure}
\begin{tikzpicture}[scale=.7, line width=1pt]
  \draw (-2,0) grid (7,1);
  \draw (1,2) grid (6,1);
  \draw (2,3) grid (6,1);
   \draw (2,4) grid (6,3);
    \draw (3,5) grid (5,3);
    \draw[->] (3,0) -- (3,7);
    \draw[->] (0,0) -- (8, 0);
    \draw[<-] (-3,0) -- (0, 0);
     \draw[color=red, line width= 2pt] (-2,0) -- (-2,1) -- (1,1) -- (1,2) -- (2,2) -- (2,4)-- (3,4) -- (3,5)-- (5,5) -- (5,4) -- (6,4) -- (6,1) -- (7,1) -- (7,0);
    \draw  (1,3) -- (1,3) node[anchor=east] {\Large {\color{ao(english)} $V_{\ell}(\lambda)$}};
    \node at (-2,1) [circle,fill, inner sep=2 pt, color=ao(english)]{};
    \node at (1,1) [circle,fill, inner sep=2 pt, color=ao(english)]{};
    \node at (1,2) [circle,fill, inner sep=2 pt, color=ao(english)]{};
    \node at (2,2) [circle,fill, inner sep=2 pt, color=ao(english)]{};
    \node at (2,4) [circle,fill, inner sep=2 pt, color=ao(english)]{};
    \node at (3,4) [circle,fill, inner sep=2 pt, color=ao(english)]{};
    \node at (5,4) [circle,fill, inner sep=2 pt, color=blue]{};
    \node at (6,4) [circle,fill, inner sep=2 pt, color=blue]{};
    \node at (6,1) [circle,fill, inner sep=2 pt, color=blue]{};
    \node at (7,1) [circle,fill, inner sep=2 pt, color=blue]{};
    \node at (5,5) [circle,fill, inner sep=2 pt, color=violet]{};
    
    \node at (3,5) [circle,fill, inner sep=2 pt, color=violet]{};
    \draw  (8,3) -- (8,3) node[anchor=east] {\Large {\color{blue} $V_{r}(\lambda)$}};
    \draw (4,5) -- (4,5) node[anchor=south] {\Large {\color{violet} $V_p(\lambda)$}};
 \end{tikzpicture}
\caption{Renormalized shape ${\color{red} \widetilde{\varphi}(\lambda)}$, left vertices {\color{ao(english)} $V_{\ell}(\lambda)$} right vertices {\color{blue} $V_r(\lambda)$}, and peak vertices {\color{violet} $V_p(\lambda)$} for the unimodal sequence $\lambda=(1,1,1,2,4,5,5,4,1)$.}
\label{vertices}
\end{figure}

Also, throughout we will use the notation $\{z^aq^n\} S(z,q)$ to denote the coefficient of $z^aq^n$ in the series $S(z,q)$.  Throughout $c$ be some locally defined constant.

\section{Proof of Theorem \ref{T:D}}\label{S:D}

We will find it easier to analyze the left part of the shape after translating into the first quadrant.  Let
$$
\overline{f_d}(x):= f_d  \left(x- \frac{\sqrt{6}}{\pi}\log(2) \right) = - \frac{\sqrt{6}}{\pi} \log \left(2e^{-\frac{\pi}{\sqrt{6}}x}-1 \right) \qquad \text{for $x \in \left[0, \frac{\sqrt{6}}{\pi}\log(2) \right)$.}
$$
We will also make use of the inverse function for $\overline{f_d}$, namely
$$
\overline{g_d}(y):= \frac{\sqrt{6}}{\pi}\left(\log(2)-\log\left(1+e^{-\frac{\pi}{\sqrt{6}}y}\right) \right) \qquad \text{for $y \in [0, \infty)$.}
$$

Let $\overline{V_{\ell}}(\lambda)$ be the left vertices after translating $\widetilde{\varphi}(\lambda)$ to the right so that the left-most vertex lies on the $y$-axis.  We want to show
\begin{equation}\label{Vllim}
\lim_{n \to \infty} \frac{1}{d(n)} \cdot \# \left\{ \lambda \in \D(n): \overline{V_{\ell}}(\lambda) \not\subset N_{\epsilon}(\overline{f_d}) \right\} = 0.
\end{equation}

We have the following inequalities.
\begin{align}
& \#\left\{ \lambda \in \D(n): V_{\ell}(\lambda) \not\subset N_{\epsilon}(f_d) \right\} \nonumber \\
&\leq \sum_{1 \leq a < \sqrt{2n}} \#\left\{\lambda \in \D(n) : \frac{1}{\sqrt{n}} (a,b) \in \overline{V_{\ell}}(\lambda),  \ \left|\overline{g_d}\left(\frac{b}{\sqrt{n}} \right) - \frac{a}{\sqrt{n}} \right| > \epsilon \right\}  \label{dsumineq2}\\
&\leq \sum_{1 \leq a < \sqrt{2n}} 2\cdot \#\left\{\lambda \in \D(n) : \lambda \ \text{has exactly $a$ left parts $\leq b$ with}  \ \left|\overline{g_d}\left(\frac{b}{\sqrt{n}} \right) - \frac{a}{\sqrt{n}} \right| > \epsilon, \right.  \nonumber \\ & \hspace{30mm}  \ \text{and the peak is $\geq b+1$} \bigg\}  \label{dsumineq3}
\end{align}
\eqref{dsumineq2} follows from the definition of $N_{\epsilon}$, and \eqref{dsumineq3} is easy to see geometrically.  After multiplying \eqref{dsumineq3} by $\frac{1}{d(n)}$, we will show that each summand is $e^{-C\sqrt{n}+ o(\sqrt{n})}$, where $C>0$ is independent of $a$.  It then follows that $$\lim_{n \to \infty} \frac{1}{d(n)} \cdot \# \left\{ \lambda \in \D(n): \overline{V_{\ell}}(\lambda) \not\subset N_{\epsilon}(\overline{f_d}) \right\} \leq \lim_{n \to \infty} 2\sqrt{2n} \cdot e^{-C\sqrt{n}+o(\sqrt{n})} = 0,$$ so \eqref{Vllim} holds.

Let $\D(q):= \sum_{m \geq 0} d(m)q^m$.  Since $d(n)$ is the $n$-th coefficient, we clearly have
\begin{equation}\label{Dcoeffconcentration}
d(n) \leq q^{-n}\D(q) \qquad \text{for $q \in (0,1)$.}
\end{equation}

The following Lemma shows that we can choose $q$ depending on $n$ that concentrates the mass of $\D(q)$ in the single term $d(n)q^n$, in the sense that after taking a logarithm, \eqref{Dcoeffconcentration} becomes an asymptotic.  (This $q$ is the unique saddle point on the positive real axis of $\left|q^{-n}\D(q)\right|$ for complex $q$.)

\begin{lemma}\label{L:D}
There exists a unique $c>0$ such that for $q= e^{-\frac{c}{\sqrt{n}}}$, we have $$\log \left(q^{-n}\D(q) \right) \sim \log d(n) = \frac{2\pi}{\sqrt{6}} \sqrt{n} + o(\sqrt{n}).$$
\end{lemma}

\begin{proof}  Recall that $\log d(n) \sim \frac{2\pi}{\sqrt{6}}\sqrt{n}$ (see \cite{BM}, Table 1).  Theorem 4.3 of \cite{BJSMR} states that $\D(e^{-t}) \sim \frac{1}{4} e^{\frac{\pi^2}{6t}}$ as $t \to 0^+.$  Letting $t=\frac{c}{\sqrt{n}}$, we see that 
$$
\log \left(e^{c\sqrt{n}} \D\left(e^{-\frac{c}{\sqrt{n}}} \right) \right) \sim \left(c+\frac{\pi^2}{6c} \right)\sqrt{n}.
$$  We take $c=\frac{\pi}{\sqrt{6}},$ and the lemma is proved.
\end{proof}

Throughout we let $c:= \frac{\pi}{\sqrt{6}}$ and $q= e^{-\frac{c}{\sqrt{n}}}$.  We will use Lemma \ref{L:D} in the form $\frac{q^{-n}\D(q)}{d(n)} \sim e^{o(\sqrt{n})}.$

Using standard combinatorial techniques, the generating function $\D(q)$ is obtained by summing over peaks as 
$$
\D(q):=\sum_{m \geq 0} d(m)q^m=\sum_{m \geq 0} q^{m+1} \prod_{j=1}^m (1+q^j)^2.
$$
Here, the two products generate the partitions to the left and right of the peak $n+1$, respectively.  Similarly, the proportion of $\lambda \in \D(n)$ with exactly $a$ left parts at most $b$ and peak at least $b+1$ is
\begin{align}
 &\{z^aq^n\} \frac{1}{d(n)}\sum_{m \geq b} q^{m+1} \prod_{j=1}^m (1+q^j)^2 \prod_{j \leq b} \frac{1+zq^j}{1+q^j}. \label{coeffineq}
\end{align}
The product with $z$'s above has the effect of replacing the original factor $1+q^j$ generating a left part $j\leq b$ with $1+zq^j$.  Importantly, it is independent of $m$, so we may factor it out and use Lemma \ref{L:D} to obtain

\begin{equation}\label{summandasym}
\{z^aq^n\} \frac{1}{d(n)}\sum_{m \geq b} q^{m+1} \prod_{j=1}^m (1+q^j)^2 \prod_{j \leq b} \frac{1+zq^j}{1+q^j} \leq \frac{q^{-n}\D(q)}{d(n)} z^{-a}\prod_{j \leq y\sqrt{n}} \frac{1+zq^j}{1+q^j} =: e^{o(\sqrt{n})} \cdot e^{U(\tau)},
\end{equation}
where we have set $b= y\sqrt{n}$ and $z=e^{\tau}$ for $\tau \in \R$ and $y \geq 0$, and
$$
U(\tau):= -\tau a + \sum_{1 \leq j \leq y\sqrt{n}} \left(\log(1+e^{\tau-c\frac{j}{\sqrt{n}}}) - \log(1+e^{-c\frac{j}{\sqrt{n}}}) \right).
$$

Note that $U(0)=0$, thus to get exponential decay in \eqref{summandasym}, we must have $\tau \neq 0$.  Taking derivatives, we have
$$
U'(\tau)=  -a + \sum_{1 \leq j \leq y\sqrt{n}} \frac{e^{\tau-\frac{cj}{\sqrt{n}}}}{1+e^{\tau-\frac{cj}{\sqrt{n}}}};
\qquad
U''(\tau)=\sum_{1 \leq j \leq y\sqrt{n}} \frac{e^{\tau-\frac{cj}{\sqrt{n}}}}{\left(1+e^{\tau-\frac{cj}{\sqrt{n}}}\right)^2}.
$$

Let $\Sigma'_{y}$ and $\Sigma''_{y}$ denote the two sums directly above.  Then multiplying by $\frac{1}{\sqrt{n}}$, we get Riemann sums for the following integrals,
\begin{equation}\label{Sigma'asym}
\frac{1}{\sqrt{n}}\Sigma'_{y} \to \int_0^{y} \frac{e^{\tau-ct}}{1+e^{\tau-ct}}dt =  \frac{1}{c} \left(\log(1+e^{\tau})- \log(1+e^{\tau-c y})\right),
\end{equation}
 and
\begin{equation}\label{Sigma''asym}
\frac{1}{\sqrt{n}}\Sigma''_{y} \to \int_0^{y} \frac{e^{\tau-ct}}{(1+e^{\tau-ct})^2}dt =  \frac{1}{c} \left( \frac{e^{\tau}}{1+e^{\tau}}- \frac{e^{\tau-c y}}{1+e^{\tau-c y}} \right).
\end{equation}

Let $\delta > 0$.  The integrands in \eqref{Sigma'asym} and \eqref{Sigma''asym} are monotonically decreasing functions of $t$; hence from integral comparison, we have, for $|\tau| < \delta$ and $y \in [0, \infty),$
$$
\left|\frac{1}{c}\left(\log(1+e^{\tau}) - \log\left(1+e^{\tau-cy}\right)\right)- \frac{1}{\sqrt{n}}\Sigma'_{y} \right| < \frac{1}{\sqrt{n}} \cdot\frac{e^{\delta}}{1+e^{\delta}},
$$
and
$$\left|\frac{1}{c} \left( \frac{e^{\tau}}{1+e^{\tau}}- \frac{e^{\tau-cy}}{1+e^{\tau-cy}} \right)-\frac{1}{\sqrt{n}}\Sigma''_{y} \right| < \frac{1}{\sqrt{n}} \cdot \frac{e^{\delta}}{(1+e^{\delta})^2}.
$$
Thus the convergence in \eqref{Sigma'asym} and \eqref{Sigma''asym} is uniform in $y \in [0, \infty)$ and $|\tau| < \delta$.

Using Taylor's Theorem, we now have
$$
U(\tau) \leq \tau U'(0) + \frac{\tau^2}{2} \sup_{|\sigma| < \delta} |U''(\sigma)| \sim \tau \sqrt{n} \left(-\frac{a}{\sqrt{n}} + \frac{1}{c}\left(\log(2)-\log(1+e^{-cy})\right) + O(\tau) \right)
$$
$$
= \tau\sqrt{n} \left(-\frac{a}{\sqrt{n}} + \overline{g_d}(y) + O(\tau) \right),
$$
where, because of uniformity in \eqref{Sigma''asym}, $O(\tau)$ does not depend on $y$.  Thus, choosing $\tau$ small in absolute value and positive or negative as needed, we get that $U(\tau) \leq -C\sqrt{n}$ for some $C>0$ that holds for all $\left(\frac{a}{\sqrt{n}},y\right)$ with $\left|-\frac{a}{\sqrt{n}} + \overline{g_d}(y) \right| \geq \epsilon.$  This proves \eqref{Vllim}. \\ \

In a similar way, we can prove
\begin{equation}\label{Vrlim}
\lim_{n \to \infty} \frac{1}{d(n)} \cdot \# \left\{ \lambda \in \D(n): \overline{\overline{V_{r}}}(\lambda) \not\subset N_{\epsilon}(\overline{\overline{f_d}}) \right\} = 0,
\end{equation}
where the right-most vertex of a sequence is fixed on the $y$-axis and $\overline{\overline{f_d}}$ is a right-translate of $f_d$ by $\frac{\sqrt{6}}{\pi}\log(2)$.

Because the left and right limit shapes in \eqref{Vllim} and \eqref{Vrlim} are {\it negative} results about 0\% of sequences, and because to obtain these results we have fixed left-most and right-most vertices on the $y$-axis, more work is needed to be able to ``glue'' these shapes together at the $y$-axis and finish the proof of Theorem \ref{T:D}.  

Now fixing the left peak vertex on the $y$-axis and noting that the peak occurs once, it is obvious that
$$
\lim_{n \to \infty} \frac{1}{d(n)} \cdot \# \left\{ \lambda \in \D(n): V_{p}(\lambda) \not\subset N_{\epsilon}(f_d) \right\} = 0,
$$
since $N_{\epsilon}$ contains an $\epsilon$-neighborhood of the $y$-axis.  Finally, since the total area under a renormalized diagram is 1, the above with \eqref{Vllim} and \eqref{Vrlim} implies Theorem \ref{T:D}.

\section{Proof of Theorem \ref{T:S}}\label{S:S}

For the proof of Theorem \ref{T:S}, we need to estimate a slightly different product, and again we will want to do our manipulations in the first quadrant.  Once left and right limit shapes are obtained, we show they can be ``glued together'' at the $y$-axis which follows from a well-known asymptotic for partitions with restricted largest part.

Throughout the proof let $c:= \frac{\pi}{\sqrt{3}}$ and $q=e^{-\frac{c}{\sqrt{n}}}$.  This $c$ is the constant needed in the following lemma, an analogue of Lemma \ref{L:D}.  Let $\ST(q):= \sum_{m \geq 0} s(m)q^m$, and recall from \cite{W2} that
$$
\ST(q) = \prod_{m\geq 1} \frac{1}{(1-q^m)^2} \cdot L(q), \qquad \text{where} \quad L(q):= \sum_{m \geq 1}(-1)^{m+1}q^{\frac{m(m+1)}{2}}.
$$

\begin{lemma}\label{L:S}
There exists a unique $c>0$ such that for $q= e^{-\frac{c}{\sqrt{n}}}$, we have $$\log \left(q^{-n}\ST(q) \right) \sim \log s(n) = \frac{2\pi}{\sqrt{3}} \sqrt{n} + o(\sqrt{n}).$$
\end{lemma}
\begin{proof}
From the well-known transformation of the Dedekind $\eta$-function (\cite{A76}, Th. 3.1), one has 
$$
\log \prod_{m \geq 1} \frac{1}{1-e^{-mt}} \sim \frac{\pi^2}{6t}, \qquad \text{as $t \to 0^+$.}
$$  
By Lemma 2 of \cite{W2} we have $\log L(e^{-t}) \sim \log \frac{1}{2}.$  Thus, letting $t= \frac{c}{\sqrt{n}},$ we have $$\log\left( e^{-c\sqrt{n}} \ST\left(e^{-\frac{c}{\sqrt{n}}}\right)\right) \sim \left(c+ \frac{\pi^2}{3c}\right)\sqrt{n}.$$  We take $c=\frac{\pi}{\sqrt{3}}$, minimizing the term on the right, and the lemma is proved.
\end{proof}

Let $\overline{f_s}(t)$ be the left-half of $f_s$ translated right into the first quadrant as follows 
$$
\overline{f_s}: \left[0, - \frac{1}{c} \log \left(1-e^{-c\epsilon}\right) \right) \mapsto [\epsilon, \infty), \qquad \overline{f_s}(x)= -\frac{1}{c} \log\left(1-e^{cx}(1-e^{-c\epsilon}) \right).
$$

We will also make use of the inverse for $\overline{f_s}$ which is 
$$
\overline{g_s}: [\epsilon, \infty) \mapsto \left[0, - \frac{1}{c} \log \left(1-e^{-c\epsilon}\right) \right), \qquad \overline{g_s}(y):= \frac{1}{c} \log \left(\frac{1-e^{-cy}}{1-e^{-c\epsilon}} \right).
$$

By Lemma \ref{L:S}, an upper bound for the proportion of the number of stacks of size $n$ with $a$ left parts that lie in $[\epsilon \sqrt{n}, y\sqrt{n}]$ and peak at least $y\sqrt{n}$+1, is
\begin{equation}\label{SVasym} 
\frac{q^{-n} \ST(q)}{s(n)} z^{-a} \prod_{ \epsilon \sqrt{n} \leq j \leq y \sqrt{n}} \frac{1-q^j}{1-zq^j} =: e^{o(\sqrt{n})} \cdot e^{U(\tau)},
\end{equation}
where $z=e^{\tau}$ for $\tau \in \R$ and
$$
U(\tau):= -\tau a + \sum_{\epsilon \sqrt{n} \leq j \leq y \sqrt{n}} \left(\log(1-e^{-c\frac{j}{\sqrt{n}}}) - \log(1-e^{\tau-c\frac{j}{\sqrt{n}}}) \right).
$$
We find the derivatives
$$
U'(\tau)=  -a + \sum_{\epsilon \sqrt{n} \leq j \leq y \sqrt{n}} \frac{e^{\tau-\frac{cj}{\sqrt{n}}}}{1-e^{\tau-\frac{cj}{\sqrt{n}}}}; \qquad U''(\tau)=\sum_{\epsilon \sqrt{n} \leq j \leq y \sqrt{n}} \frac{e^{\tau-\frac{cj}{\sqrt{n}}}}{\left(1-e^{\tau-\frac{cj}{\sqrt{n}}}\right)^2}.
$$
Let $\Sigma'_{y}$ and $\Sigma''_{y}$ denote the two sums directly above.  Then multiplying by $\frac{1}{\sqrt{n}}$, we get Riemann sums for the following integrals:
\begin{equation}\label{bSigma'asym} \frac{1}{\sqrt{n}}\Sigma'_{y} \to \int_{\epsilon}^{y} \frac{e^{\tau-ct}}{1-e^{\tau-ct}}dt =  \frac{1}{c} \left(-\log(1-e^{\tau-c\epsilon})+ \log(1-e^{\tau-cy})\right)= \frac{1}{c} \log \left(\frac{1-e^{\tau-cy}}{1-e^{\tau-c\epsilon}} \right),\end{equation}
 and
\begin{equation}\label{bSigma''asym}\frac{1}{\sqrt{n}}\Sigma''_{y} \to \int_{\epsilon}^{y} \frac{e^{\tau-ct}}{(1-e^{\tau-ct})^2}dt =  \frac{1}{c} \left( \frac{e^{\tau-c\epsilon}}{1-e^{\tau-c\epsilon}}- \frac{e^{\tau-cy}}{1-e^{\tau-cy}} \right). \end{equation}
Let $\delta > 0$.  From integral comparison and monotonicity of the integrand, it is easy to see that, for $|\tau| < \delta$ and $y \in [\epsilon, \infty),$
$$\left|\frac{1}{c}\left(-\log(1-e^{\tau-c\epsilon}) + \log\left(1-e^{\tau-cy}\right)\right)- \frac{1}{\sqrt{n}}\Sigma'_{y} \right| < \frac{1}{\sqrt{n}} \cdot\frac{e^{\delta-c\epsilon}}{1-e^{\delta-c\epsilon}},$$ and
$$\left|\frac{1}{c} \left( \frac{e^{\tau}}{1-e^{\tau-c\epsilon}}- \frac{e^{\tau-cy}}{1-e^{\tau-cs}} \right)-\frac{1}{\sqrt{n}}\Sigma''_{y} \right| < \frac{1}{\sqrt{n}} \cdot \frac{e^{\delta-c\epsilon}}{(1-e^{\delta-c\epsilon})^2}.$$  Thus the convergence in \eqref{bSigma'asym} and \eqref{bSigma''asym} is uniform in $y \in [\epsilon, \infty)$ and $|\tau| < \delta$.  Using Taylor's Theorem, we now have
\begin{align*} U(\tau) \leq \tau U'(0) + \frac{\tau^2}{2} \sup_{|\sigma| < \delta} |U''(\sigma)| &\sim \tau \sqrt{n} \left(-\frac{a}{\sqrt{n}} +  \frac{1}{c} \log \left(\frac{1-e^{-cy}}{1-e^{-c\epsilon}} \right) + O(\tau) \right) \\ &\sim \tau \sqrt{n} \left(-\frac{a}{\sqrt{n}} +  \overline{g_s}(y)  + O(\tau) \right) \end{align*} where, because of uniformity in \eqref{bSigma''asym}, $O(\tau)$ does not depend on $y$.  Thus, when $\left|-\frac{a}{\sqrt{n}}-\overline{g_s}(y)\right| \geq \epsilon$, we conclude, as in the proof of Theorem \ref{T:D}, that $$\lim_{n \to \infty} \frac{1}{s(n)} \cdot \#\{\lambda \in \ST(n): \overline{V_{\ell}}(\lambda) \not\subset N_{\epsilon}(\overline{f_s}) \} =0,$$ where the left vertices $\overline{V_{\ell}}(\lambda)$ have been translated so that the first left part that is at least $\epsilon \sqrt{n}$ has been placed directly to the left of the $y$-axis.

By symmetry, we can also say that $$\lim_{n \to \infty} \frac{1}{s(n)} \cdot \#\{\lambda \in \ST(n): \overline{\overline{V_{r}}}(\lambda) \not\subset N_{\epsilon}(\overline{\overline{f_s}}) \} =0,$$ for similar right-translates of $f_s$ and right-vertices.  However, we cannot, as in Section \ref{S:D}, use an area argument to immediately conclude Theorem \ref{T:S}.

For strongly unimodal sequences, the fact that parts on either side are distinct and the total area under a diagram is 1 forced a limit shape from the {\it negative} result that on average $0\%$ of left (resp. right) halves of diagrams are {\it not} near a left (resp. right) limit shape.  But we do not have this forcing here because of the possibility of repetition of peaks so that the peak vertices no longer lie in an $\epsilon$-neighborhood of the $y$-axis.  Nevertheless, we will show that peaks of almost all diagrams are $\omega(\sqrt{n})$, which implies they occur with multiplicity $o(\sqrt{n})$ and so 
$$
\lim_{n \to \infty} \frac{1}{s(n)}\#\left\{\lambda \in \ST(n) : V_p(\lambda) \not\subset N_{\epsilon}(f_s)\right\} = 0.
$$
This will then conclude the proof.

\begin{lemma}\label{L:Speaks}
Let $t > 0$ be an arbitrary fixed constant.  If $k=t\sqrt{n}$ and $s_{\leq k}(n)$ denotes the number of unimodal sequences of size $n$ in which peaks are size at most $k$, then $$\lim_{n \to \infty} \frac{s_{\leq k}(n)}{s(n)} = 0.$$
\end{lemma}

\begin{proof}[Proof of Lemma \ref{L:Speaks}] Let $\ST_k(n)$ denote the set of stacks of size $n$ in which peaks have size $k$.  Let $\mathscr{P}_{\leq k}(n)$ denote the set of partitions of $n$ into parts $\leq k$.  Let $s_k(n)$ and $p_{\leq k}(n)$ be the cardinality of these sets, respectively.  Then we have an injection 
\begin{equation}\label{stinj} 
\ST_k(n) \hookrightarrow \bigcup_{m=0}^{n} \mathscr{P}_{\leq k}(m) \times \mathscr{P}_{\leq k}(n-m),  
\end{equation} 
given by cutting a stack $\lambda \in \ST_k(n)$ in half directly right of the left-most peak.\footnote{We can be more precise about the image in \eqref{stinj}, but we will not need to be.}  Thus, we may write  
\begin{equation}\label{sksum} 
\begin{split}  
s_k(n) &\leq \sum_{m=0}^n p_{\leq k}(m)  p_{\leq k} (n-m) \\ &\leq 2\sum_{0 \leq m \leq \epsilon \cdot n} p_{\leq k}(m)p_{\leq k}(n-m) + \sum_{\epsilon \cdot n \leq m \leq (1-\epsilon)n} p_{\leq k}(m)  p_{\leq k} (n-m) \\ & =: 2\Sigma_1 + \Sigma_2 
\end{split} 
\end{equation} 
for some $\epsilon=\epsilon(t) \in \left(0,\frac{1}{2}\right)$ to be specified later.
Asymptotics for $p_{\leq k}(n)$ when $k= t\sqrt{n}$ were given first by Szekeres \cite{S53},  reformulated and reproved by Canfield \cite{C} and later by Romik \cite{R}.  From  Romik's formulation, 
$$
p_{\leq k}(n) \ll e^{H(t)\sqrt{n}},
$$ where
\begin{align*}
H(t) &= 2\alpha(t) - t \log \left(1-e^{-t\alpha(t)}\right) \\
\alpha:&[0, \infty) \to \left[0, \frac{\pi}{\sqrt{6}} \right) \ \text{defined by} \ \alpha(t)^2 = \text{Li}_2\left(1-e^{-t\alpha(t)}\right)
\end{align*}

We now show that $\alpha:[0, \infty) \to \left[0, \frac{\pi}{\sqrt{6}} \right)$ is strictly increasing; in particular, $\alpha(t)$ is well-defined as above.  One finds
$$
\alpha'(t)= \frac{t\alpha(t) }{2(e^{t\alpha(t)}-1)-t^2}.
$$
The numerator is positive for $t > 0$, so it remains to show that the denominator is positive for $t > 0$.  We will actually show
\begin{equation}\label{E:alpharatiopositive}
\frac{t^2}{e^{t\alpha(t)}-1} < 1, \qquad \text{for $ t > 0$.}
\end{equation}

Following Canfield (\cite{C}, Comment 19), we have
\begin{align*}
\alpha(t)^2&= \text{Li}_2\left(1-e^{-t\alpha(t)}\right) \\
&= -\int_0^{1-e^{-t\alpha(t)}} \frac{\log(1-z)}{z}dz \\
&= \int_0^{t\alpha(t)} \frac{u}{e^u-1}du & \text{(substituting $z=1-e^{-u}$)}\\
&> t\alpha(t) \cdot \frac{t\alpha(t)}{e^{t\alpha(t)}-1}. & \text{(since the integrand is increasing)}
\end{align*}
From this, \eqref{E:alpharatiopositive} follows, so $\alpha$ is strictly increasing.  Next, it may be checked that $H'(t)=-\log\left(1-e^{-t\alpha(t)}\right) > 0$, so that $H$ is strictly increasing.  Furthermore, we have 
$$
\lim_{t \to \infty} H(t) = \lim_{t \to \infty} 2 \alpha(t) = \pi \sqrt{\frac{2}{3}}.
$$

Returning to \eqref{sksum}, we may write any $m \in [\epsilon n , (1-\epsilon)n]$ as $sn$ for some $s \in [\epsilon, 1-\epsilon]$.  Thus, $$\Sigma_2 \ll n \exp\left(\sqrt{n} \sup_{s \in [\epsilon, 1-\epsilon]}\left(\sqrt{s}H\left(\frac{t}{\sqrt{s}}\right)+\sqrt{1-s}H\left(\frac{t}{\sqrt{1-s}}\right) \right) \right).$$  Since $H$ is strictly increasing to $\pi\sqrt{\frac{2}{3}}$, for any fixed $t$, $\epsilon$ and all $s\in [\epsilon, 1-\epsilon]$, there is a $B=B_{\epsilon, t} < \pi \sqrt{\frac{2}{3}}$ such that $H\left(\frac{t}{\sqrt{s}}\right), H\left(\frac{t}{\sqrt{1-s}}\right) \leq B$.  Hence, we may bound the above by $$n \exp\left(\sqrt{n} \sup_{s \in [\epsilon, 1-\epsilon]} \left(\sqrt{s}\cdot B + \sqrt{1-s} \cdot B \right)\right) = n \exp\left(\sqrt{n} \cdot \sqrt{2}\cdot B \right).$$

Now, $$2\Sigma_1 \ll n p_{\leq k}(\floor{\epsilon n}) p_{\leq k}(n) \ll n \exp\left(\sqrt{n}\left(\sqrt{\epsilon}H\left(\frac{t}{\sqrt{\epsilon}}\right) + \pi \sqrt{\frac{2}{3}}\right) \right).$$  Since $H$ is bounded, we may choose $\epsilon=\epsilon(t)$ so that $$C:= \sqrt{\epsilon}H\left(\frac{t}{\sqrt{\epsilon}}\right) + \pi \sqrt{\frac{2}{3}} < \pi \frac{2}{\sqrt{3}}.$$

Thus, altogether we have $$s_{\leq k}(n) \leq ns_{k}(n) \ll n^2 \exp\left(\sqrt{n} \cdot C \right) + n^2 \exp\left(\sqrt{n} \cdot B\sqrt{2} \right),$$ where $C, B\sqrt{2} < \pi \frac{2}{\sqrt{3}}.$  Recalling that $s(n) \sim \frac{1}{2^3 \cdot 3^{3/4} \cdot n^{5/4}} e^{\pi \frac{2}{\sqrt{3}} \sqrt{n}}$ (\cite{W2}, Th. 2), we have finished the proof of Lemma \ref{L:Speaks}. \end{proof}

By our earlier observations, the the proof of Theorem \ref{T:S} is now complete.

\section{Proof of Theorem \ref{T:DM}}\label{S:DM}

Here, the derivations of left and right limit shapes are similar, respectively, to those in Sections \ref{S:D} and \ref{S:S}.  Thus, in this section we content ourselves to proving an analogue of Lemmas \ref{L:D} and \ref{L:S} and to showing that peaks are $\omega(\sqrt{n})$ on average; as in Section \ref{S:S}, this is necessary to avoid the possibility of a degenerate limit shape.

Let $\D_m(q):= \sum_{m \geq 0} dm(m)q^m.$  As required by our technique, the next lemma shows that there is a choice of $q$ so that $$\frac{q^{-n}\D_m(q)}{dm(n)}=e^{o(\sqrt{n})}.$$
\begin{lemma}\label{L:DM}
There exists $c>0$ such that for $q= e^{-\frac{c}{\sqrt{n}}}$, we have $$\log \left(q^{-n}\D_m(q) \right) \sim \log dm(n) = \pi \sqrt{n} + o(\sqrt{n}).$$
\end{lemma}

\begin{proof}  By equation (3.3) and Theorem 1.3 of \cite{BM}, we have $\log dm(n) \sim \pi \sqrt{n}$, and $$\log \left(e^{c\sqrt{n}} \D_m\left(e^{-\frac{c}{\sqrt{n}}} \right) \right) \sim \left(c+\frac{\pi^2}{4c} \right)\sqrt{n}.$$
We take $c=\frac{\pi}{2},$ and the lemma is proved.
\end{proof}

With Lemma \ref{L:DM} in hand, we may derive left and right limit shapes as in Sections \ref{S:D} and \ref{S:S}, respectively.  Thus, the proof will be completed by the following lemma, which shows that peaks are $\omega(\sqrt{n})$ on average.

\begin{lemma}\label{L:DMpeaks}
Let $t > 2$ be an arbitrary fixed constant.  If $k=t\sqrt{n}$ and $dm_{\leq k}(n)$ denotes the number of stacks of size $n$ in which the peak is at most $k$, then 
$$
\lim_{n \to \infty} \frac{dm_{\leq k}(n)}{dm(n)} = 0.
$$
\end{lemma}

\begin{remark}
Since $t_1 \leq t_2$ implies $dm_{\leq t_1\sqrt{n}}(n) \leq dm_{\leq t_2\sqrt{n}}(n)$, the conclusion of Lemma \ref{L:DMpeaks} holds with any $t \geq 0$.
\end{remark}

\begin{proof}[Proof of Lemma \ref{L:DMpeaks}]

Let $\mathscr{Q}_{\leq k}(n)$ and $q_{\leq k}(n)$ be, respectively, the set and number of distinct-parts partitions of $n$ whose largest part is at most $k$.  As in Section \ref{S:S}, we define a map 
\begin{equation}\label{dminj} 
\D_{m,k}(n) \hookrightarrow \bigcup_{m=0}^{n} \mathscr{P}_{\leq k}(m) \times \mathscr{Q}_{\leq k}(n-m),  
\end{equation}
by sending the peak and left parts to a distinct partition, and by sending the right parts to an unrestricted partition.  Thus,
\begin{equation}\label{dmsum}
\begin{split}
dm_{k}(n) &\leq \sum_{m \leq \epsilon n} \left(q_{\leq k}(m)p_{\leq k}(n-m) + q_{\leq k}(n-m)p_{\leq k}(m)\right) + \sum_{\epsilon n \leq m \leq (1-\epsilon)n} q_{\leq k}(m) p_{\leq k}(n-m)  \\ &=: \Sigma_1 + \Sigma_2.
\end{split}
\end{equation}

Szekeres found an asymptotic for distinct parts partitions of size $n$ when the number of parts is bounded by $t\sqrt{n}$ \cite{S51}.  To the best of our knowledge, however, the asymptotic we need, namely when part {\it sizes} are bounded $t\sqrt{n}$, has only recently been found by the author \cite{B}.  We stress that these counts are not the same due to lack of conjugation symmetry on the set of distinct parts partitions.  Thus, in a manner similar to \cite{R}, we have found that, for $t>2$, $$q_{\leq k}(n) \ll e^{B(t)\sqrt{n}},$$ where $B(t)$ is a strictly increasing function with $\lim_{t\to \infty} B(t)=\frac{\pi}{\sqrt{3}}$ \cite{B}.  For $m \in [\epsilon n, (1-\epsilon)n]$, we will write $m=sn,$ for $s \in [\epsilon, 1-\epsilon]$.  Thus, as in Section \ref{S:S}, 
$$
\Sigma_2 \ll n \exp\left(\sqrt{n} \sup_{s \in [\epsilon, 1-\epsilon]} \left(\sqrt{s}H\left(\frac{t}{\sqrt{s}}\right) + \sqrt{1-s}B\left(\frac{t}{\sqrt{1-s}}\right)\right)\right) 
$$ $$
\ll n\exp\left(\sqrt{n} \cdot C \sup_{s \in [\epsilon , (1-\epsilon)]} \left(\sqrt{2s}+ \sqrt{1-s}\right)\right),
$$ where $C < \frac{\pi}{\sqrt{3}}.$  Thus, 
$$
\Sigma_2 \ll n \exp\left(\sqrt{n} \cdot C \sqrt{3} \right) =o\left(\pi \exp(\sqrt{n})\right).
$$

Now, $$
\Sigma_1 \ll n q_{\leq k}\left(\floor{\epsilon n}\right) p_{\leq k}(n) + nq_{\leq k}(n) p_{\leq k}\left(\floor{\epsilon n}\right)) 
$$ $$
\ll n \exp\left(\sqrt{n} \left(\sqrt{\epsilon} B\left(\frac{t}{\sqrt{\epsilon}}\right) + \pi\sqrt{\frac{2}{3}}\right)\right) + n\exp\left(\sqrt{n}\left(\frac{\pi}{\sqrt{6}}+\sqrt{\epsilon}H\left(\frac{t}{\sqrt{\epsilon}}\right)\right)\right),
$$
where we have used the asymptotic formulas for $p(n)$ and $q(n)$ (\cite{A98}, Th. 6.2).  Since $H$ is bounded, we may choose $\epsilon = \epsilon(t)$ so that $\Sigma_1 \ll o\left(\exp(\pi \sqrt{n})\right).$  

Finally, since $dm(n) \sim \frac{1}{16n}e^{\pi \sqrt{n}}$ (\cite{BM}, Th. 1.3), we have $dm_{\leq k}(n)=o(dm(n))$ as required.
\end{proof}

This concludes the proof of Theorem \ref{T:DM}.

\end{document}